\documentclass[a4paper,10pt,reqno]{amsart}

\usepackage{amsfonts,amsmath,mathrsfs,amssymb,amsthm,amscd,latexsym,amstext,amsxtra}

\usepackage[usenames]{color}
\usepackage[all]{xy}
\usepackage{enumerate}
\xyoption{all}
\usepackage{srcltx} %per fer cerca inversa
\usepackage{amsthm, amsmath, amssymb,latexsym}
\usepackage{amsfonts,epsfig,amscd}
\usepackage[]{fontenc}
\usepackage{xy}
\usepackage{enumerate}
\usepackage[]{fontenc}
\usepackage{xy}
\newtheorem{theorem}{Theorem}[section]
\newtheorem{lemma}[theorem]{Lemma}

\newtheorem{proposition}[theorem]{Proposition}
\newtheorem{remark}[theorem]{Remark}
\newtheorem{definition}[theorem]{Definition}

\newcommand{\nc}{\newcommand}
\nc{\cH}{{\mathcal H}}
\nc{\cA}{{\mathcal A}}
\nc{\cG}{{\mathcal G}}
\nc{\cC}{{\mathcal C}}
\nc{\cO}{{\mathcal O}}
\nc{\cI}{{\mathcal I}}
\nc{\cB}{{\mathcal B}}
\nc{\cY}{{\mathcal Y}}
\nc{\cK}{{\mathcal K}}
\nc{\cX}{{\mathcal X}}
\nc{\cS}{{\mathcal S}}
\nc{\cE}{{\mathcal E}}
\nc{\cF}{{\mathcal F}}
\nc{\cZ}{{\mathcal Z}}
\nc{\cQ}{{\mathcal Q}}
\nc{\cN}{{\mathcal N}}
\nc{\cP}{{\mathcal P}}
\nc{\cL}{{\mathcal L}}
\nc{\cM}{{\mathcal M}}
\nc{\cT}{{\mathcal T}}
\nc{\cW}{{\mathcal W}}
\nc{\cU}{{\mathcal U}}
\nc{\cJ}{{\mathcal J}}
\nc{\cV}{{\mathcal V}}
\nc{\bH}{{\mathbb H}}
\nc{\bA}{{\mathbb A}}
\nc{\bG}{{\mathbb G}}
\nc{\bC}{{\mathbb C}}
\nc{\bO}{{\mathbb O}}
\nc{\bI}{{\mathbb I}}
\nc{\bB}{{\mathbb B}}
\nc{\bY}{{\mathbb Y}}
\nc{\bK}{{\mathbb K}}
\nc{\bX}{{\mathbb X}}
\nc{\bS}{{\mathbb S}}
\nc{\bE}{{\mathbb E}}
\nc{\bF}{{\mathbb F}}
\nc{\bZ}{{\mathbb Z}}
\nc{\bQ}{{\mathbb Q}}
\nc{\bN}{{\mathbb N}}
\nc{\bP}{{\mathbb P}}
\nc{\bL}{{\mathbb L}}
\nc{\bM}{{\mathbb M}}
\nc{\bT}{{\mathbb T}}
\nc{\bW}{{\mathbb W}}
\nc{\bU}{{\mathbb U}}
\nc{\bD}{{\mathbb D}}
\nc{\bJ}{{\mathbb J}}
\nc{\bV}{{\mathbb V}}
\nc{\bbZ}{{\mathbb Z}}
\nc{\bR}{{\mathbb R}}
\nc{\fr}{{\rightarrow}}
\nc{\dbar}{{\bar{\partial}}}
\nc{\co}{{\nabla}}

\nc{\cu}{{\overlineline{\nabla}}}
\begin{document}

\title{Some results on deformations of sections of vector bundles.}
%\titlerunning{Deformations of sections} 
\author[A. Castorena]{Abel Castorena}
\address{Centro de Ciencias Matem\'aticas (Universidad Nacional Auton\'oma de M\'exico, Campus Morelia); Apartado Postal 61-3 (Xangari), 58089 Morelia, Michoac\'an}
\email{abel@matmor.unam.mx}
\author[G.P. Pirola]{Gian Pietro Pirola}
\address{Dipartimento di Matematica, Universit\`a degli Studi di Pavia, Via Ferrata 1, 27100 Pavia, Italy}
\email{gianpietro.pirola@unipv.it}
\thanks{The first named author was partially supported by sabbatical Grants 232811(CONACyT, M\'exico) and PASPA-DGAPA (UNAM, M\'exico), and was partially supported as Visiting Professor by INdAM(Istituto Nazionale di Alta Matematica ``F. Severi").\\ 
The second named author is partially supported by INdAM (GNSAGA); PRIN 2012 \emph{``Moduli, strutture geometriche e loro applicazioni''} and FAR 2014 (PV) \emph{``Variet\`a algebriche, calcolo algebrico, grafi orientati e topologici''.}}

\maketitle
\begin{abstract} Let $E$ be a vector bundle on a smooth complex projective variety $X$. We study the family of sections $s_t\in H^0(E\otimes L_t)$ where $L_t\in Pic^0(X)$ is a family of topologically trivial line bundle and $L_0=\cO_X,$ that is, we study deformations of $s=s_0$.
By applying the approximation theorem of Artin \cite{A} we give a transversality condition that generalizes the semi-regularity of an effective Cartier divisor. Moreover, we obtain another proof of the Severi-Kodaira-Spencer theorem \cite{Bh}. We apply our results to give a lower bound to the continuous rank of a vector bundle as defined by Miguel Barja \cite{B} and a proof of a piece of the generic vanishing theorems \cite{GL1} and \cite{GL2} for the canonical bundle. We extend also to higher dimension a result given in \cite{MPP1}  on the base locus of the paracanonical base locus for surfaces.  
\end{abstract}

\section{Introduction}

Let $X$ be a smooth complex projective variety of dimension $n$. Let $P=Pic^0(X)$ be the Picard variety of 
$X$. We recall that the tangent space to $P$ at any point is isomorphic to $H^1(X,\cO_X)$. The irregularity of $X$ is $q(X):=\dim H^1(X,\cO_X)=\dim P$, and $X$ is said to be irregular if $q(X)>0$. Given a vector bundle $E$ on $X$ and $L\in P$, the action $(L,E)\to E\otimes L=E(L)$ defines a family of vector bundles parametrized by $P$. This paper is devoted to study the Brill-Noether loci:
$$\cW^k(E):=\{L \in P: h^0(E(L))\geq k+1\}$$ where $h^i(E)=\dim H^i(X,E)$. Set $\cW^0(E)=\cW(E)$.
As in \cite{MPP2} we would like to give some conditions that implies 
that $\dim \cW(E)>0$. This is a deformation problem for the  sections of  $E$. Given $s\in H^0(X,E)$ and $\zeta\in H^1(X,\cO_X)$, denote by $\zeta\cdot s$ its cup product, then the first order condition to deform $s$ in the direction $\zeta$ is $\zeta\cdot s=0\in H^1(X,E)$, and in general there are higher obstructions. Our first result  (See section $2$)  is that they vanishes under a transversality hypothesis that generalizes, in our context, the semi-regularity condition and the Koszul condition 
\cite{MPP1} borrowed from the work of M. Green and R. Lazarsfeld {\cite{GL1}. For any $\zeta\in H^1(X,\cO_X)$, define the complex
$$H^0(X,E)\stackrel{\cdot\zeta}\to H^1(X,E)\stackrel{\cdot\zeta}\to H^2(X,E),$$
\noindent and set $H^1(X,E)_\zeta= \ker(\zeta : H^1(X,E) \to H^2(X,E))$. From now on, we will denote by $\Delta\subset\Bbb C$ a complex disk with $0\in\Delta$. Our first result is :

\begin{proposition}{\bf(Transversality.)} \label{trans}
Let $s\in H^0(X,E)$ and $\zeta\in H^1(X,\cO_X)$ be elements such that $\zeta\cdot s=0$ in $H^1(X,E)$.
If $H^1(X,E)_\zeta=s\cdot H^1(X,\cO_X)+\zeta\cdot H^0(X,E)$ then there exists an analytic curve $ \gamma:\Delta \to P$ such that $\gamma(0)=\cO_X$ and $\gamma'(0) =\zeta$, and there exists a family of sections $s(t) \in H^0(X,E(\gamma(t)))$ such that $s(0)=s$. In other words, the direction $\zeta$ is not obstructed for $s$. Moreover we may assume that the image of  $\gamma$ contains an open set of an algebraic curve. In particular $\dim \cW(E)>0$.
\end{proposition}
To prove our result we first solve the problem in the formal setting,  then we apply the powerful approximation theorem of Artin (See \cite{A}, Theorem 1.2) that assure the existence of a convergent solution and finally  the algebraic version, since the scheme $\cW(E)$ is algebraic.  

\vskip2mm

An important case of Proposition $\ref{trans}$ is when $E$ is a line bundle and this is in connection with the semi-regularity. In fact, we remark in section 2.1 that an effective Cartier divisor $D$ which is semi-regular (See definition 2.4 below) satisfies $H^1(X,E)=H^1(X,E)_{\zeta}=s\cdot H^1(X,\mathcal O_X)$, where $E=\mathcal O_X(D)$, $s\in H^0(X,E)$ is a section such that $D$ is the the zero-locus of $s$ and $\zeta\cdot s=0$ for some $\zeta$. The semi-regularity gives also information on the infinitesimal structure of the Hilbert scheme $\mathcal H_{X,D}$ that parametrizes effective divisors on $X$ with fixed Chern class (See section 2). To be more precise, the Severi-Kodaira-Spencer theorem asserts that if $D$ is a semi-regular divisor, then $\mathcal H_{X,D}$ is smooth at $[D]$ (See e.g. \cite{Bh} Theorem 1.2 part (i),  \cite{M}, p. 157). In this context we give another proof of this theorem by applying the transversality condition of Proposition \ref{trans}, in particular the semi-regularity is an open condition on the locus $\cW(E)\subset\text{Pic}(X)$.

\vskip3mm

To illustrate our second result (See section 3), following Miguel Barja (See \cite{B}), we let  
\begin{equation} r_c=r_c(E)=\min_{L\in P} h^0(E(L)),\label{conrank}
\end{equation}
be {\it{the continuous rank of $E$.}} Note that $r_c(E)>0 \iff \cW(E)=P$. We will see that $r_c\geq h^0(E)-h^1(E).$ 
This inequality should be well-known and it certainly will sound familiar to any expert on deformation theory.
A  stronger result will be given using the Koszul condition.  We have :
\begin{proposition} For any $\zeta\in H^1(X,\cO_X)$, set $h^1(E)_\zeta=\dim H^1(X,E)_\zeta.$  It holds $r_c(E)\geq h^0(E)-h^1(E)_\zeta\geq h^0(E)-h^1(E).$ \label{crank} \end{proposition}

When $X$ is a smooth curve we have $h^1(X,E)=h^1(X,E)_\zeta$, and by Riemann-Roch theorem the Proposition \ref{crank} gives $r_c\geq \chi(E)=c_1(E)-e(g-1)$. Thus for $e=1$, that is,  when $E$ is a line bundle we obtain $\cW(E)=P$ for $c_1(E)\geq g,$  a version of the Jacobi inversion theorem.

\vskip2mm
In  section $4$ we consider the case when $E=\wedge^k\Omega^1_X=\Omega_X^k$ is the sheaf of holomorphic $k-$forms on $X$. Let $K_X=\Omega^n_X$ be the canonical sheaf on $X$. Using the $\partial\bar{\partial}-$ lemma (See \cite{V}, Proposition 6.17) we obtain a quick proof of a piece of the generic vanishing theorem of Green-Lazarsfeld.  We give it by completeness, with not claim of originality. 
The approach is very similar to the one given by Clemens and Hacon in \cite{CH}, but again we use the formal approximation theorem of Artin to show that the higher order obstructions vanish.

\vskip2mm

In section $5$ we extend to higher dimensional varieties the result given in (\cite{MPP1}, Theorem 1.3) on the fixed part of the paracanonical base locus of surfaces. To explain this, let $P^{\kappa}$ be the isomorphism classes of line bundles algebraically equivalent to $K_X.$ The morphism $L\to L\otimes K^{-1}_X$ gives an identification of $P^\kappa$ and  $P$. Let  $\text{Div}^{\kappa}(X)$ be the projective variety that parametrizes effective divisors with class $\kappa$, $\text{Div}^{\kappa}(X)$ is the {\em paracanonical system} of $X$. We consider the map 
 $\rho:\text{Div}^{\kappa}(X)\to P^\kappa, D\mapsto\mathcal O_X(D).$ The fiber $\rho^{-1}([L])$
 is isomorphic to the linear system $|L|$. If $r_c(K_X)>0$, there is a unique irreducible component $\text{Div}^{\kappa}(X)_{\text{main}}$
  of $\text{Div}^{\kappa}(X)$ that maps surjectively onto $P^\kappa$. The variety
  \begin{equation}\text{Div}^{\kappa}(X)_{\text{main}}\label{main}\end{equation}
  is called {\it{main component of }}$\text{Div}^{\kappa}(X)$ of the paracanonical system.

\vskip1mm
We define the {\it{Albanese variety}} $A$ of $X$, to be the dual of $P: A=Pic^0(P)$. Let $a: X\to A$ be the {\it Albanese morphism} (See \cite{MPP1}). We say that $X$  has {\it maximal Albanese dimension}  if $\dim X=\dim a(X)$ and it is of {\it Albanese general type} if in addition $q=q(X)>\dim X$. We recall that generic vanishing theorems for the canonical bundle implies that $r_C(K_X)=\chi(K_X) $ for $X$ of maximal Albanese dimension (See \cite{GL1}). If in addition $X$ is not fibred onto a variety $Y$ with $0<\text{dim}(Y)<\text{dim}(X)$ and with any smooth projective model $\widetilde Y$ of $Y$ of Albanese general type, then $\chi(K_X)\geq q-\dim X$. 
(See \cite{PP}, Theorem A).

We define respectively the base locus $Z_K$ of the canonical system and $Z_\kappa$ the base locus of the main paracanonical component by:
$$Z_K= \{p \in X: p\in D \  \forall D\in |K_X|\}.$$
$$Z_\kappa=\{p \in X:  p\in D \ \forall D\in \text{Div}^{\kappa}(X)_{\text{main}}\}.$$

\vskip1mm
Our result is the following:

\begin{proposition}\label{zeri}
Let $X$ be a variety of Albanese general type with $\chi(K_X)>0$, then  $Z_\kappa\subset Z_{K}.$ \end{proposition}
\vskip2mm

\noindent{\bf{Acknowledgments.}} The authors give thanks to Rita Pardini and Miguel Angel Barja for fruitful discussions and suggestions. Special  thanks are due to the referee for valuable comments which helped us to improve the manuscript.

The first named author give thanks to the Department of Mathematics of the University of Pavia for their warm hospitality and the use of their resources during a sabbatic year.
 
%%%%%%%%%%%%%%%%%%%%%%%%%%%%%%%%%%%%%%%%%%%%%%%%%%%%%%%%%%%%%%%%%%%
%%%%%%%%%%%%%%%%%%%%%%%%%%%%%%%%%%%%%%%%%%%%%%%%%%%%%%%%%%%%%%%%%%%

\section{Deformations of vector bundles and transversality}\label{defo}
Let  $E\to X$ be a holomorphic vector bundle of rank $e$. Let $\cA^{p,q}(E)$ be the sheaf of $\cC^ \infty $ forms of type $(p,q)$ with values on $E$, $\cC^\infty(E)=\cA^{0,0}(E)$. We have an operator $$\bar \partial_E: \cA^{p,q}(E) \to \cA^{p,q+1}(E)$$ such that for any open subset $U\subset X$, 
 
 $$U\to E(U)= {\ker \bar\partial _E}{_{|U}} : \cC^\infty(E)(U)\to \cA^{0,1}(E)(U)$$ 
 gives the sheaf of holomorphic sections of $E$. The correspondence $E \leftrightarrow \bar\partial _E$ with the integrability condition is given by $\bar\partial^2 _E=0$ (cf. e.g. \cite{W}). A deformation of $E$ can be seen as a deformation of $\dbar_E$, that is, an operator $\dbar _E+T:\cA^{p,q}(E) \to \cA^{p,q+1}(E)$ satisfying  the integrability condition $(\bar\partial_E+T)^2=0$.  Let  $\gamma:\Delta \to P$ be an analytic curve in $P$ and assume that $\gamma(0)=\cO_X$. The operators associated to the deformation $E(\gamma(t))=E\otimes\gamma(t)$, $\gamma(t)\in P$, have a considerably simple linearization. To see this, consider the exponential map $\exp: H^1(X,\mathcal O_X)\to P,$ write $\gamma(t)=\exp(\alpha(t))$, where $\alpha:\Delta \to H^1(X,\mathcal O_X)$. We have $\alpha(t)= [v(t)]$, where $v(t)\in \cA^{0,1}(X)$ is a $(0,1)$-form such that $\bar{\partial}(v(t))=0$. A canonical way to represent $\alpha(t)$ is to use the harmonic representative, that is, $v(t)=\overline{\omega}(t)$ where $\omega(t)\in H^0(X,\Omega^1_X)$. The corresponding family of operators on $E(\gamma(t))$ can be given by (See \cite{GL2}, p. 92)
\begin{equation} \bar \partial_{E(\gamma(t))}=\bar \partial_{E}+\wedge v(t). \label{op} \end{equation}
$$ \bar \partial_{E(\gamma(t))}(s)=\bar \partial_{E}(s)+s\wedge v(t).$$

\noindent The integrability conditions $ \bar\partial_{E(\gamma(t))}^2=0$ is satisfied since $\bar \partial v(t)=0$ and 
$v(t)\wedge v(t)=0.$ 
\vskip2mm

Given $\zeta_1=[v_1]\in H^1(X,\cO_X)$ and $f_0\in H^0(X,E)$, the cup product $[v_1\cdot f_0] \in H^1(X,E)$ 
gives the first order obstruction  to deform $f_0$ in the $\zeta_1$-direction. In fact, if $[v_1\cdot f_0]=0$, we can find $f_1\in\cC^\infty(E)$ such that $v_1\cdot f_0+\dbar_Ef_1=0$, then 

$$(\dbar_E+tv_1)(f_0+tf_1)=t^2 (v_1\cdot f_1)=0 \mod \ t^2.$$

Now we take sequences $f_i\in H^0(X,E)$, $v_i\in \cA^{0,1}(\cO_X)$ satisfying $\dbar v_i=0$ and
 $[v_i]\in H^1(X,\cO_X)$. We write the {\em formal} equation:
\begin{align}\label{formale}
(\dbar_E+tv_1+t^2v_2+\cdots+t^nv_n+\cdots)\cdot(f_0+tf_1+t^2f_2+\cdots+t^nf_n+\cdots)=0,
\end{align}
\noindent in other words, such that the following is satisfied: $$\dbar_Ef_0=0,\hskip3mm \dbar_Ef_1+v_1f_0=0,\hskip3mm \dbar_Ef_2+v_1f_1+v_2f_0=0,
$$ 

\noindent and in general:

\begin{align} \label{formale2} \dbar_Ef_n+v_1f_{n-1}+\cdots+v_kf_{n-k}+\cdots+v_nf_0=0.\end{align}
\begin{definition} Fix $f_0\in H^0(X,E)$ and $\zeta_1=[v_1]\in H^1(X,\cO_X)$. We say that the couple $(\zeta_1,f_0)$ 
allows a formal solution, if there exists a sequence $(v_i,f_i)$ satisfying the formal equation (\ref{formale}).
We say moreover that the formal solution is linear if in (\ref{formale}) we can assume $v_i=0$ for $i>1.$
\label{nobs}
\end{definition}

The following proposition is a consequence of the  Artin's approximation theorem (See \cite{A}, Theorem 1.2). 

\begin{proposition}\label{convergence} Let $f_0\in H^0(X,E)$ and  $\zeta_1\in H^1(X,\cO_X)$ be such that $(\zeta_1,f_0)$ allows a formal solution. Then there is an analytic curve  
$\gamma:\Delta \to P$ such that $\gamma(0)=\cO_X$ and $\gamma'(0) =\zeta_1$, and a family of sections $s(t) \in H^0(X,E(\gamma(t)))$ such that $s(0)=f_0$. In other words the direction $\zeta_1$ is not obstructed for $s$. In particular $\dim \cW(E)>0$ and we may assume that the image of  $\gamma$ contains an open set of an algebraic curve.\end{proposition}
\begin{proof}
We take a Zariski open set $W$ of $P$ that contains the origin $\cO_X$, that is, there is an embedding $\pi:W\hookrightarrow\bC^N$ such that the ideal of $\pi( \cW(E))\subset \bC^N$ is generated by polynomial equations $\{F_j\}.$ The formal curve $$\tilde\gamma(t)=\exp (\sum_{i=1}^\infty\ \zeta_i t^i)$$ satisfies the  equations $F_i(\tilde\gamma(t))\equiv 0$. By (\cite{A}, Theorem 1.2 for $c=1$) we can find an analytic curve $\gamma:\Delta \to P,$  $\gamma(0)=\cO_X$ and $\gamma'(0) =\zeta_1$. The last sentence follows again since  the  $F_j$ are polynomials. \qed\end{proof}

 We give the following
\begin{lemma}  Suppose that for $\zeta_i=[v_i]$ and $i\leq n$, we have solved equation (\ref{formale2}) in the $n$-th step, that is, 
$$
\dbar_Ef_n+v_1f_{n-1}+\cdots+v_kf_{n-k}+\cdots +v_nf_0=0 \label{ex}.
$$
Let  $y=v_1f_n+v_2f_{n-1}+\cdots+v_kf_{n-k+1}+v_nf_1$ be, then
\begin{enumerate} 
\item $\dbar_E (y)=0,$
\item   $v_1\wedge y\in \text{Im}\ \dbar_E $
\end{enumerate}
\end{lemma}
\begin{proof}Since $\dbar v_i=0$, then $\dbar_E(v_i f)=v_i\wedge\dbar_E(f)$. 
Define $d=\dbar_E +tv_1+\cdots +t^nv_n$. 
This is an operator $d: \cA^{p,q}(E)[t]\to \mathcal A^{p,q+1}(E)[t]$ satisfying $d^2=0$, where $\mathcal A^{p,q}(E)[t]:=\cA^{p,q}(E) \otimes \bC[t]$ and $\bC[t]$ is the ring of complex polynomials. Taking $f=f_0+tf_1+\cdots +t^nf_n\in \cA^{0,0}(E)[t]=\cC^\infty(E)[t]$,
we get $df=0 \mod t^n$ by our assumption. Making some computations we have that
$$df=t^{n+1}y+ t^{n+2}y_1+\cdots$$
where $y=v_1f_n+v_2f_{n-1}+\cdots+v_kf_{n-k+1}+v_nf_1$.
Now we have $d^2(f)=0$, thus
$$0=d(y+ty_1+\cdots)=(\dbar_E +tv_1+\cdots +t^nv_n)(y+ty_1+\cdots)$$ the first terms gives
$$\dbar_E(y)=0 \  , \  v_1\wedge y +\dbar_E (y_1)=0.$$
\qed\end{proof}

Now we prove the Proposition \ref{trans}:

\begin{proof} Suppose that $(\zeta_1=[v_1],f_0)$ satisfies the transversal condition: $\zeta_1H^0(X,E)+f_0H^1(X,\cO_X)=H^1(X,E)_{\zeta_1}$.  Moreover, assume that we have solved the equation (\ref{formale2}) in the $n$-th step, that is,
$$\dbar_Ef_n+\sum_{i=1}^n v_i f_{n-i}=0.$$ We would like to solve the $(n+1)$th step. The proposition will follow by induction. Define $y=\sum^n_{i=1}v_if_{n+1-i},$  we know that $\dbar_E y=0,$ $\zeta_1\cdot [y]=[v_1\cdot y]=0\in H^2(X,E)$ by the above lemma.
Let $[y]\in H^1(X,E)_{\zeta_1}$ be its class. By the transversal hypothesis there are $\zeta_{n+1}\in H^1(X,\cO_X)$ and $\tilde s\in H^0(X,E)$ such that
$$[y]= \zeta_1\cdot \tilde s+\zeta_{n+1}f_0,$$ that is, choosing a representative  $v_{n+1}$ for $\zeta_{n+1},$
$y=v_{n+1}f_0+ v_1\tilde s+\dbar_Eg.$   
Letting $f'_{n}=f_n-\tilde s$ we solve again the  $n$-th step equation 
$$\dbar (f_n-\tilde s)  +\sum_{i=1}^n v_i f_{n-i}=\sum_{i=1}^n v_i f_{n-i}=0.$$ 
Then letting $f_{n+1}=-g$ we solve also the $(n+1)$-th step:
$$\dbar (f_{n+1}) + v_1(f_n-\tilde s)+v_2f_{n-1}+\dots +v_nf_1+ v_{n+1}f_0.$$
We see that $f_0,...,f_{n-1},f'_{n}=f_n-\tilde s,f_{n+1}$ stabilizes in the first $n-1$ elements.
Then we can find a formal solution and by (\ref{convergence}) we prove our result.
\qed\end{proof}

%%%%%%%%%%%%%%%%%%%%%%%%%%%%%%%%%%%%%%%%%%%%%%%%%%%%%%%%%%%%%%%%%%%%%%%%%%%%%%%%%%%%%%%%%%%%%%%%%%%%%%%%%%%%%%%%%%%%%%%%%%%%%%%%%%%%%%%%%%%%%%%

\subsection{Semi-regularity} Let $X$ be a smooth complex projective variety. The {\it{semi-regularity}} of an effective Cartier divisor $D\subset X$ is an important example which satisfies the transversality condition of Proposition \ref{trans} for $E=L=\mathcal O_X(D)$, and is related also with the problem to determinate when a couple $(s,v)$ is non-obstructed in the sense of Proposition \ref{convergence}. To see this, let $s\in H^0(X,L)$ be a section such that $(s)_0=D$ is the zero-locus of $s$ and let $N_D=L|_D$  the normal sheaf of $D$ in $X$. The normal bundle exact sequence $0\to\mathcal O_X\xrightarrow{\cdot s}L\to N_D\to 0$ induces an exact sequence in cohomology
\begin{align}\cdots\to H^0(D, N_D)\xrightarrow{\delta_D}H^1(X,\mathcal O_X)\xrightarrow{\cdot s}H^1(X,L)\to H^1(D,N_D)\xrightarrow{\delta^1}H^2(X,\mathcal O_X)\to\cdots\end{align}

\begin{definition} The divisor $D$ is said to be {\it{semi-regular}} if $H^1(X,L)\to H^1(D,N_D)$ is the zero-map.\end{definition}

Let $\mathcal H_{X,D}$ be the Hilbert scheme  that parametrizes effective divisors $B$ on $X$ such that $c_1(\mathcal O_X(B))=c_1(\mathcal O_X(D))$. Is well known that first-order deformations 
$\tilde D\subset X\times\text{Spec}(\frac{k[t]}{(t^2)})$ of $D$ in $X$ are in one-to-one correspondence with global sections $\tilde s$ in $H^0(D, N_D)$, where $H^0(D, N_D)$ is the Zariski tangent space $T_{[D]}(\mathcal H_{D,X})$ to $\mathcal H_{X,D}$ at  $[D]$ (See \cite{S}, Proposition 3.2.1). The obstruction to lifting $\tilde D$ to a deformation over
$\text{Spec}(\frac{k[t]}{(t^n)})$ for some $n\geq 3$ determines a cohomological class $ob(\tilde s)$ in $H^1(D,N_D)$ (See \cite{S}, p. 131). If $ob(\tilde s)\neq 0$, then $\tilde s$ is called obstructed, otherwise it is unobstructed. The set of obstructions define a subspace in $H^1(D,N_D)$, and this obstruction space is contained in the kernel of the semi-regular map $H^1(D,N_D)\xrightarrow{\delta^1}H^2(X,\mathcal O_X)$ (See e.g. \cite{Bh}). In particular when $\delta^1$ is injective the section $\tilde s$ is non-obstructed.
\vskip1mm

\noindent Note that for $t\in H^0(X,N_D)=T_{[D]}(\mathcal H_{X,D})$ and $v=\delta_D(t)\in\text{Im}(\delta_D)\subseteq H^1(X,\mathcal O_X)$ we have that $s\cdot v=0$. From the exact sequence (6), semiregularity for $D$ is equivalent to have $H^1(X,L)=s\cdot H^1(X,\mathcal O_X)$, and by properties of cap product we have $v\cdot H^1(X,L)=v\cdot(s\cdot H^1(X,\mathcal O_X))=\{0\}\subset H^2(X,L)$, then $H^1(X,L)_v=\text{Ker}(H^1(X,L)\xrightarrow{v}H^2(X,L))=s\cdot H^1(X.\mathcal O_X)=H^1(X,L)$, thus the transversal condition of Proposition \ref{trans} is satisfied for the couple $(s, v)$.

\vskip2mm

The following theorem was first claimed by Severi and proved by Kodaira-Spencer (See e.g. \cite{Bh} Theorem 1.2 part (i), \cite{M}, p. 157). We give a proof of such theorem by applying the transversality criteria of Proposition \ref{trans}.

\vskip1mm

\begin{theorem}(Severi-Kodaira-Spencer). Let $X$ be a complex smooth projective variety. Let $D\subset X$ be an effective Cartier divisor and set $L=\mathcal O_X(D)$. If $D$ is semiregular, then Hilbert scheme $\mathcal H_{X,D}$ is smooth at $[D]$ of the expected dimension $h^0(X,L)-h^1(X,L)-1+h^1(X,\mathcal O_X)$.
\end{theorem}
\begin{proof}As we have seen, the semiregularity is equivalent to the fact that the map $H^1(X,\mathcal O_X)\xrightarrow{\cdot s}H^1(X,L)$ is surjective, and in this case we have the transversality condition of Proposition \ref{trans} for couples $(s,w)$ for every $w\in\text{Im}(\delta_D)=\text{Ker}(H^1(X,\mathcal O_X)\xrightarrow{\cdot s}H^1(X,L))$. Since semiregularity is an open condition on the sections $s$, the dimension of the Hilbert scheme $\mathcal H_{X,D}$ is precisely $h^0(X,L)-1+\text{dim Im}(\delta_D)$ and the Severi-Kodaira-Spencer theorem follows.\qed\end{proof}

%%%%%%%%%%%%%%%%%%%%%%%%%%%%%%%%%%%%%%%%%%%%%%%%%%%%%%%%%%%%%%%%%%%
%%%%%%%%%%%%%%%%%%%%%%%%%%%%%%%%%%%%%%%%%%%%%%%%%%%%%%%%%%%%%%%%%%%
\subsection{Examples.} 

\noindent We give some examples where the Transversal condition of Proposition \ref{trans} is satisfied. In these examples $C$ will denote a non-hyperelliptic curve of genus $g\geq 3$.

\vskip2mm

\noindent (a).- Let $p\in C$ be a general point and let $u$ be a general extension in $\text{Ext}^1(K_C,\mathcal O_C(p))\simeq H^1(C,T_C(p))$, then the correponding coboundary map 
$\partial_u: H^0(C,K_C)\to H^1(C,\mathcal O_C(p))=(H^0(C,K_C(-p)))^{\vee}$ is surjective, so there exists on $C$ a rank two vector bundle $E$ fitting in an exact sequence $0\to\mathcal O_C(p)\to E\to K_C\to 0$ such that $h^0(C,E)=2$. 
\vskip1mm
\noindent{\it{Claim.}} We can find two linearly independent sections $s, t$ in $H^0(C,E)$ such that $s(p)=0$ and $t(p)\neq 0$. 

\noindent{\it{Proof of Claim.}} We will show that for the general extension in $\text{Ext}^1(K_C(-p),\mathcal O_C)$ the corresponding coboundary map $H^0(C,K_C(-p))\to H^1(C,\mathcal O_C)$ is injective. 

\noindent Let $\eta\in\text{Ext}^1(K_C(-p),\mathcal O_C)\simeq H^1(C,T_C(p))$ and let $0\neq\omega\in H^0(C,K_C(-p))$ such that 
the cup product $\eta\cdot\omega=0$. Let $\Lambda_{\omega}:=\{\eta\in H^1(C,T_C(p)): \omega\cdot\eta=0\}$ and set $Z:=Z(\omega)$ the divisor of $\omega$ which is of degree $2g-3$ and let $\Bbb P:=\Bbb P^{g-2}=\Bbb P(H^0(C,K_C(p)))$. Tensoring by $T_C(p)$ the exact sequence $0\to\mathcal O_C\xrightarrow{\omega}K_C(-p)\to K_C(-p)|_Z\to 0$ we have $0\to T_C(p)\xrightarrow{\omega}\mathcal O_C(p)\to\mathcal O_C(p)|_{Z}\to 0$. Taking cohomology one has that $\text{dim}(\Bbb P(\Lambda_{\omega}))=2g-5$ and $\text{dim Span}(\bigcup_{\omega\in\Bbb P}\Bbb P(\Lambda_{\omega}))\leq 2g-5+g-2=3g-7<3g-5=\text{dim}(\Bbb P(H^1(C,T_C(p))))$. This proves the claim.

\vskip2mm

\noindent Now consider the following diagram 

$$
\xymatrix{
& & & 0\ar[d] & \\
& 0\ar[d] & &\Bbb C_p=\mathcal O_p(p)\ar[d]&\\
0\ar[r] & \mathcal O_C\ar[d]\ar[r] & E\ar[r]^{\pi_1}\ar[d]_{Id} & E/\mathcal O_C\ar[d]\ar[r] & 0\\
0\ar[r] & \mathcal O_C(p)\ar[d]_{\delta_p}\ar[r]^{f^0}& E\ar[r]^{\pi_2} & K_C\ar[d]\ar[r]  & 0\\
&\mathcal O_p(p)\ar[d] &  & 0\\
& 0 
}(\star)
$$
\noindent Let $v=v_p$ be a non-zero element in $\text{Im}(\delta_p)\subset H^1(C,\mathcal O_C)$. From $(\star)$ we have an square diagram in cohomology :
$$
\xymatrix{
& H^0(C,\mathcal O_C(p))\ar[d]^{f^0}\ar[r]^{\cdot v} & H^1(C,\mathcal O_C(p))\ar[d]^{f^1}\\
& H^0(C,E)\ar[d]^{\pi_2}\ar[r]^{\cdot v} & H^1(C,E)\ar[d]^{\wr\wr}\\
& H^0(C,K_C)\ar[r]^{\cdot v} & H^1(C,K_C) &
}
$$
\noindent  where $f^0, f^1$ are induced maps in cohomology. Let $u\in H^0(C,\mathcal O_C(p))$ be a non-zero generator. Since $f^0$ is injective, the section $u\in H^0(C,\mathcal O_C(p))$ can be identified with $f^0(u)=s\in H^0(C,E)$, then $v\cdot s=0=v\cdot f^0(s)=f^1(v\cdot u)\in H^1(C,E)$. For the other side, the section $t\in H^0(C,E)$ satisfies that $w:=\pi_2(t)\neq 0$ in $H^0(C,K_C)$ and $0\neq v\cdot w\in H^1(C,K_C)$, then $v\cdot t\neq 0$ in $H^1(C, E)$, that is, we see that $v\cdot H^0(C,E)\neq 0$. Since $h^1(C,E)=1$, we have the transversal condition $v\cdot H^0(C,E)+s\cdot H^1(C,\mathcal O_C)=H^1(C,E)$.

\vskip1mm

\noindent (b).- Let $S=S^2(C)$ be the second symmetric product of $C$. For every $p\in S$ we have a natural divisor $X_p=\{p+q:q\in C\}$ in $S$. Let $N^1(S)_{\Bbb Z}$ be the N\'eron-Severi group of $S$ of numerically equivalent classes of divisors on $C^{(2)}$. Let  $x$ be the class of $X_p$ in $N^1(S)_{\Bbb Z}$. This class is an ample class and is independent of of the point $p$. Let $\delta'$ be the class in $N^1(S)_{\Bbb Z}$ of the diagonal divisor $\Delta=\{q+q:q\in C\}\subset S$.  If $C$ has very general moduli, it is known that  $N^1(S)_{\Bbb Z}$ is generated by the numerical equivalence classes $x$ and 
$\delta=\frac{\delta'}{2}$ (See e.g. \cite{ACGH}, p. 359). The intersection numbers between these numerical classes are $(x,x)=x^2=1,(x,\delta)=x\cdot\delta=1, (\delta,\delta)=(\delta)^2=1-g$.
\vskip1mm
\noindent For a non-hyperelliptic curve $C$ of genus three consider the divisor $D:=K_S-x=3x-\delta$ and the line bundle $L=\mathcal O_S(D)$. Let $s_D$ be the section of $L$ such that $D$ is the zero-locus of $s_D$. We have that $h^0(S,L)=1$, $\chi(L)=0$ and $h^2(S,L)=h^0(S,\mathcal O_S(x))=1$, then $h^1(S,L)=h^0(S,L)+1=2$. For the normal bundle $N_D$ we have $\text{deg}(N_D)=D^2=1$, so $h^0(N_D)=1$. Since $h^0(S,L)=h^0(S,\mathcal O_S)$, we have $0\to H^0(S, N_D)\xrightarrow{\delta} H^1(S,\mathcal O_S)\xrightarrow{\cdot s_D}H^1(S,L)\to H^1(S,N_D)\to\cdots$. Let $0\neq w\in H^0(N_D)$ be and let $v=\delta(w)$, we have that $v\cdot s_D=0$. Since $\cdot s_D$ is surjective, $D$ is semiregular and the transversal condition $v\cdot H^0(S,L)+s_D\cdot H^1(S,\mathcal O_S)=H^1(S,L)$ is clear.

%%%%%%%%%%%%%%%%%%%%%%%%%%%%%%%%%%%%%%%%%%%%%%%%%%%%%%%%%%%%%%%%%%%%%%%%%%%%%%%%%%%%%%%%%%%%%%%%%%%%%%%%%%%%%%%%%%%%%%%%%%%%%%%%%%%%%%%%%%%%%%%%%%%%%%%%%%%%%%%%%%%%%%%%%%%%%
\section{Continuous rank}
Now we study the continuous rank of a vector bundle defined $E$ as in (\ref{conrank}).
Fixing $\zeta\in H^1(X,\cO_X)$, we first define $$W_{\zeta}:=\{f\in H^0(X,E): (\zeta, f)\ \text{allows a formal linear solution}\}.$$ 

\noindent We remark that $W_{\zeta}$ is a subvector space of $H^0(X,E)$ because we consider only formal linear operators (see definition \ref{nobs}).

We will give an estimate of the dimension of $W_{\zeta}$. We recall that $H^1(X,E)_\zeta$ has been defined as the kernel  of the map $\zeta:  H^1(X,E)\to H^2(X,E)$
and $h^1(E)_\zeta=\dim H^1(X,E)_\zeta$ its dimension.
We consider the lines in $H^1(X,\cO_X)$ given by
$t\mapsto t\zeta$ and the corresponding  curves $\gamma(t)=\exp(t\zeta)$ in $P.$
If $v$ is a representative of $\zeta$ we also consider the operator
$\dbar_E+tv$.

\begin{lemma} $\dim W_\zeta\geq h^0(E)-h^1(E)_\zeta$. \end{lemma}

\begin{proof} 

For  the operator $\dbar_E+tv$  the equations (\ref{formale}) become
\begin{align}\dbar_E f_n+vf_{n-1}=0.\end{align} Lemma \ref{ex} gives just 
$\dbar_E( v\cdot f_i)=v\wedge v \cdot f_i=0.$

\noindent{\it{First step.}} Consider the linear map $$v^{(1)}: H^0(X,E)\to H^1(X,E)_\zeta, $$
$$f_0\mapsto[vf_0].$$  

\noindent This is just the cup product map.

Let $ K^0_1(v):=\ker(v^{(1)})$, $K^1_1(v):=\text{Coker}(v^{(1)})=H^1(X,E)_\zeta/\text{Im}(v^{(1)})$.  
Since $\dim \ker(v^{(1)})=h^0(X,E)-\dim \text{Im}(v^{(1)})$, then  
$$\dim K^0_1(v)-\dim K^1_1(v)=h^0(X,E)-h^1(X,E)_\zeta.$$
\noindent{\it{Second step.}} We have that for $f_0\in K^0_1(v)$, there is $f_1$ such that $\dbar_E(f_1)+vf_0=0.$ We define 
$$v^{(2)}:K^0_1(v)\to K^1_1(v),\ f_0\mapsto [vf_1].$$ 
The map $v^{(2)}$ is well defined: 
\begin{enumerate}
\item $\dbar_E (vf_1)=0$; $v\wedge vf_1=0$; $[vf_1]\in H^1(X,E)_\zeta.$
\item if $\dbar_E (g)=\dbar_E(f_1)=-vf_0$ then $\dbar_E(f_1-g)=0$  $f_1-g=s\in H^0(X,E)$;  we have 
$vf_1-vg=vs=0$ in $K^1_1(v)$ since $vs\in \text{Im} (v^{(1)})$.
\end{enumerate}

\vskip2mm

\noindent  Now if $f_0\in \ker v^{(2)}$, there is  $f_2\in \cC^\infty(E)$ such that $vf_1=-\dbar_Ef_2 - vs$, where $s\in H^0(X,E)$,
then we have $\dbar_Ef_2+v(f_1+s)=0$. Let $f'_1=f_1+s$, then  
$$(\dbar_E+tv)\cdot (f_0+t f'_1+t^2f_2)=0 \mod \ t^3.$$
We reset the notation by $f_1=f'_1$. Let $K^0_2(v)=\ker v^{(2)}$ and $ K^1_2(v)=\text{Coker} (v^{(2)})$, we have
$$\dim K^0_2(v)- \dim K^1_2(v)=\dim K^0_1(v)- \dim K^1_1(v)=h^0(E)-h^1(E)_\zeta.$$

\vskip2mm

\noindent{\it{Inductive Step.}} For $i\geq 2$, we define inductively the following linear maps: 

\vskip2mm

\noindent Let $K^0_i(v):=\ker(v^{(i)}), K^1_i(v):=\text{Coker} (v^{(i)})$ and  define
 recursively
$$v^{(i+1)}:K^0_i(v)\to K^1_i(v),\hskip5mm  f_0\mapsto [vf_i].$$ 

\noindent We have that $$\dim K^0_{i+1}(v)-\dim K^1_{i+1}(v)=\dim K^0_i(v)-\dim K^1_i(v)=h^0(E)-h^1(E)_\zeta.$$ 

As in step $2$, the condition to solve (\ref{formale2}) at the level $n$ is that $f_0\in \ker v^{(n)}$.
In fact if $v^n(f_0)=0$ we can find $s\in K^0_{n-1}(v)\subset H^0(X,E)$ such that $vf_{n}=-\dbar_Ef_{n+1}-vs$, then letting $f'_n =f_n+s$ we get 
$\dbar_Ef_{n+1}+vf'_n=0$ when $\dbar_Ef_n=\dbar_Ef'_n=vf_{n-1}$. Note that we only modify the $n-$term of the sequence $f_n$. 

\noindent Since $K^0_{i+1}(v)\subset K^0_i(v)$ this process stop in a finite number of steps, so we can find an integer $m$ such that the linear map $v^{(m+1)}:K^0_m(v)\to K^1_m(v)$ is the zero map. Then we can find a vector space $K^0_m(v) \subset H^0(X,E)$ such that for any $f_0\in K^0_m(v)$, $(\zeta,f_0)$ allows a linear formal solution. It follows that $K^0_m(v)\subset W_\zeta$. Note that 
$$\dim W_\zeta \geq \dim K^0_m(v)\geq \dim K^0_m(v)-\dim K^1_m(v)=h^0(E)-h^1(E)_\zeta.$$
\end{proof}

\noindent Now the proof of the Proposition \ref{crank} is as follows : 

\begin{proof} By semicontinuity of $h^1(E)_\zeta$ as $\zeta$ varies in $H^1(X,\cO_X)$, it is enough to prove the result for general $\zeta.$ For a general element $\zeta \in H^1(X,\cO_X)$, the analytic curve $\gamma(t)=\exp (t\zeta)$ is Zariski dense in $P$. Letting $k=h^0(E)-h^1(E)_\zeta-1$ we get that for all $t$, $\gamma(t)\in \cW(E)^k.$ In fact, first we note that  $h^0(E)\geq k+1$, that is, $\gamma(0)=\cO_X\in \cW(E)^k$, so taking an affine open set $U$ of $P$ and polynomial equations $F_j$ for the scheme $\cW(E)^k\cap U$, we get that
$F_j(\gamma(t))=0\mod t^n$ because we can  deform a space of dimension $\geq k+1$  $\mod t^n$. Since $F_j(\gamma(t))$ is holomorphic and $F_i(\gamma(t))=0 \mod t^n$ for all $n$ we get $F_i(\gamma(t))\equiv 0$, that is, $\gamma(t)\in \cW(E)^k.$  Since the image of $\gamma$ is Zariski dense we get that $\cW(E)^k=P.$ This implies $k\leq r_C(E)$.\qed\end{proof}

\begin{remark} In the previous proof, the Artin approximation theorem has not been fully used since  $\gamma(t )=\exp (t\zeta)$ is already analytic, using it one can show directly that the sections of the vector bundle deform along  $\gamma(t).$  This gives  again $\gamma(t)\in \cW(E)^k.$ \end{remark}

%%%%%%%%%%%%%%%%%%%%%%%%%%%%%%%%%%%%%%%%%%%%%%%%%%%%%%%%%%%%%%%%%%%%%%%%%%%%%%%%%%%%%%%%%%%%%%%%%%%%%%%%%%%%%%%%%%%%%%%%%%%%%%%%%%%%%%%%%%%%%%%%%%%%%%%%%%%%%%%%%%%%%%%%%%%%%%%%%%%%%%%%%%%%%

\section{Generic vanishing toward the $\partial\dbar$ lemma}

In this section we give a proof of a piece of the generic vanishing of  Green and Lazarsfeld (See \cite{GL2}).
\begin{lemma} \label{dbar}
Let $E=\Omega^k_X$ be and $\omega=f_0\in H^0(X,E)=H^{k,0}(X)$ a section. 
Let  $\zeta\in H^1(X,\cO_X)$ be and assume that $\zeta\cdot \omega=0
\in H^1(X,E)=H^{k,1}(X)$, then $(\zeta,\omega)$ admits a formal linear solution (See definition \ref{nobs}).\end{lemma}

\begin{proof}  Since $H^1(X,\mathcal O_X)\simeq H^{0,1}(X)\simeq\mathcal H^{0,1}(X)$ we can take a representative of
 $[v]=\zeta$  to be harmonic, so $\partial v=0$ and $\dbar v=0$. Assume that $[v\wedge \omega]=0$ in $H^1(X,\Omega^k_X)\simeq H_{\dbar}^{k,1}(X)$, then there exists $\alpha\in\text{Im}(\dbar:\cA^{k,0}(X)\to\cA^{k,1}(X))$ such that 
 $v\wedge f_0=\dbar \alpha$. We have also that $\partial (v\wedge f_0)=\partial(v)\wedge f_0+v\wedge\partial \alpha=0+0=0$, then $v\cdot f_0\in\text{ker}(\partial:\cA^{k,1}(X)\to\cA^{k+1,1}(X))$. Applying the 
 $\partial\dbar$-lemma (See e.g. \cite{V}, Proposition 6.17) we can find $g_1\in \cC^\infty(\Omega_X^{k-1})$ such that
$$v\cdot f_0=-\dbar \partial(g_1).$$ 
Let  $f_1=\partial g_1$ be, we have that $\dbar f_1+v\wedge f_0=0$. Suppose that $f_i=\partial g_i$ such that for $i=1,...,n$, we have $0=\dbar f_i+v\wedge f_{i-1}= \dbar \partial g_i+\partial v\wedge g_{i-1}$. We perform the  $(n+1)$th step. Consider $\theta_n=\partial v\wedge g_n$. We have $\theta_n=\partial ( v\wedge g_n)$ and $\dbar(\theta_n)=-v\dbar\partial g_{n}=-v\wedge v\wedge \partial g_{n-1}=0$, then by $\partial\dbar$-lemma we have $\theta_n\in\ker(\dbar)\cap \text{Im}(\partial)=\text{Im}(\partial\dbar)$, thus there is $g_{n+1}$ such that $\theta_n=\dbar \partial(g_{n+1})$. Taking $f_{n+1}=\partial g_{n+1}$ we have 
$$(\dbar+tv)(f_0+  \sum_{i=1}^\infty  t^i \partial g_i)=0.$$
\qed\end{proof}

Arguing as before we get
\begin{proposition} (Green-Lazarsfeld) Let $\omega\in H^0(X,\Omega_X^k)$ and $\zeta\in H^1(X,\cO_X)$\label{gl} be and assume that $\zeta\cdot \omega=0\in H^1(X,\Omega_X^k)$. Define $\gamma: \bC\to P$, 
$\gamma(t) =\exp t\zeta$, then there is a family of sections $\omega(t)\in H^0(X,\Omega_X^k(\gamma(t)))$ along $\gamma(t)$ such that $\omega(0)=\omega$. \end{proposition}

\begin{remark} The proof is the same of the one given for the Proposition \ref{trans}. Following Green and Lazarsfed we note that the Zariski 
closure of $\gamma(t)$ is an abelian subvariety of $P.$ If $\zeta$ is linearly very generic, that is, is not contained in a countable union of proper vector subspaces of $H^1(X,\cO_X),$ then $\overline{\gamma(t)}=P.$  Assume $k=n=\dim X$, then $\Omega_X^n=K_X$ and
$\omega\neq 0$. Let $D(\omega)$ be the divisor of $\omega$. If $\zeta \cdot \omega=0$ and $\zeta$ is linearly very generic we have that $D(\omega)\in\text{Div}^{\kappa}(X)_{\text main}$ (See (\ref{main})), the main component of the paracanonical system.\end{remark}

%%%%%%%%%%%%%%%%%%%%%%%%%%%%%%%%%%%%%%%%%%%%%%%%%%%%%%%%%%%%%%%%%%%%%%%%%%%%%%%%%
\section{The base locus of the paracanonical system} 

In this section we will prove  Proposition \ref{zeri}. We will use again an antilinear isomorphism $H^1(X,\mathcal O_X)\simeq\overline{H^0(X,\Omega^1_X)}$. Thus, given $v\in H^1(X,\mathcal O_X)$, there exists $\beta\in H^0(X,\Omega_X^1)$ such that $[\overline{\beta}]=v$. 
\noindent
Consider the  hermitian pairing $<,>H^0(X,K_X)\times H^0(X,K_X)\to\Bbb C$ given by
\begin{equation} <\omega_1,\omega_2>=(-i)^n\int_X\omega_1\wedge\overline{\omega}_2.\label{her}\end{equation}

For $\beta\in H^0(X,\Omega^1_X)$ we have the cup product maps

$$ H^0(X,\Omega_X^{n-1})\xrightarrow{\wedge\beta} H^0(X,\Omega_X^n)=H^0(X,K_X).$$ 
$$H^0(X,K_X)\xrightarrow{\cdot\overline{\beta}} H^1(X,K_X)$$

\begin{lemma} For every $\beta\in H^0(X,\Omega^1_X)$ we have $\text{Im}(\wedge\beta)^\perp=\ker(\overline{\beta}).$ \label{ort} \end{lemma}
\begin{proof} We have

\noindent $\theta\in\text{Im}(\wedge\beta)^\perp\iff 0=\int_X \theta\wedge \overline{\alpha\wedge\beta}=\int_X(\theta\wedge\overline{\beta})\wedge\overline{\alpha},\ \forall \alpha\in H^0(X,\Omega_X^{n-1})\stackrel{{\text duality}}\iff\theta\wedge\overline \beta=0\in H^{1}(X,\Omega_X^{n-1})\iff\theta\in \ker\overline\beta$.\qed\end{proof}

\noindent Let $\bP(H^0(X,K_X))=|K_X|$ and  $\bP=\bP(H^1(X,\cO_X))$  the projective spaces of $H^0(X,K_X)$ and $H^1(X,\cO_X)$ respectively. Inside $|K|\times \bP$ consider the incidence locus $U\subset |K_X|\times \bP$: 
$$U=\{((s),(\zeta)): \zeta\cdot s=0\}.$$

We have that $U$ is the locus parametrizing pairs $(s,\zeta)$ such that $s\in\ker\zeta$. We have seen in Proposition \ref{gl} that there are not higher obstructions to deform $s$ in the $\zeta$-direction. Let $\pi_1: U\to |K_X|$ and  $\pi_2: U\to \bP$ be the projections.  We remark that the fibers of $\pi_2$ are the projective spaces corresponding to $\ker(\overline\beta)$. As a consequence of the Generic Vanishing Theorems of Green-Lazarsfeld (See \cite{GL1}, \cite{GL2} and \cite {MPP1}, Lemma 4.2) we have:
\begin{proposition}
Assume that $X$ is a variety of Albanese general type with $\chi(K_X)>0$. We have that $\pi_2(U)=\bP(H^1(X,\cO_X))$. Let $U_{main}$ be the unique irreducible component of $U$ that dominates $\bP(H^1(X,\cO_X))$, then  
$\dim(U_{main})=\chi(K_X)+q-2.$ 
  
\end{proposition}

\begin{definition}
We call the variety $Y=\pi_1(U_{main})\subset |K_X|$
the locus of deformable canonical divisors. It consists in fact of divisors that deform to the main component of the paracanonical system $Y= \text{Div}^{\kappa}(X)_{\text main}\cap |K_X|$. (See \ref{main}).
\end{definition}

\begin{proposition}
 Assume that $X$ is of Albanese general type with $\chi(K_X)>0$. Let $Y$ be the locus of deformable canonical divisors of $X$, then $Y$ is non degenerate, that is, is not contained in any proper linear subspace of $|K_X|$.
\end{proposition}
\begin{proof} Define $W\subset H^0(X,K_X)$ as the subspace $$W={\text span}_{(s)\in Y}\{ s\}.$$
We have  to show that $W=H^0(X,K_X)$. Suppose by contradiction that $W\subsetneq H^0(X,K_X)$, then there exists $\theta$ orthogonal to $\ker(\overline{\beta})$ where $(\overline{\beta})$ is generic, $\overline{\beta}\in H^1(X,\cO_X)$.
\vskip1mm

\noindent By lemma (\ref{ort}), we have that $\theta\in\ker(\overline{\beta})^{\perp}=\text{Im}(\beta)$. This means that for generic $\beta$, $$\theta=\beta\wedge\theta_{\beta}$$ for some $\theta_{\beta}\in H^0(X,\Omega_X^{n-1})$. Since this is a closed condition, then  
$\theta\in \text{Im}(\beta)$ for all $\beta.$ We can find  a point $p\in X$, $\theta(p)\neq 0$, so we consider  the evaluation map 
$$H^0(X,\Omega^1_X)\otimes\mathcal O_X\xrightarrow{ev_p}  (\Omega^1_X)_p.$$ 

Since $q>n$, there exists $\beta\neq 0$ such that $\beta(p)=0$. Since $\theta=\beta\wedge\theta_1$ for some $\theta_1$, then $\theta(p)=\beta(p)\wedge\theta_1(p)=0$ which is a contradiction, then $W=H^0(X,K_X).$
\end{proof} 
\noindent Now we prove the Proposition \ref{zeri} :

\begin{proof} Consider the base  locus $Z_\kappa$ of  the main component of the paracanonical system:
$$Z_\kappa=\{p \in X:  p\in D, \ \forall D\in \text{Div}^{\kappa}(X)_{\text{main}}\}.$$  
It follows that $Z_\kappa\subset Z=\{p\in X: p\in D(s), \  \forall (s)\in Y\}.$ Since $Y$ is non-degenerate generates $H^0(X,K_X)$, then we have $Z=\{p\in X: s(p)=0,\ \forall s \in H^0(X,K_X)\}=Z_K$.\qed\end{proof}

\end{document}